\documentclass[runningheads]{llncs}
\usepackage{graphicx}
\usepackage{breakcites}
\usepackage{amsfonts}
\usepackage{amsmath,amssymb}
\usepackage{algorithm}
\usepackage{algpseudocode}
\usepackage{pifont}
\usepackage{float}
\usepackage{graphicx}
\usepackage{caption}
\usepackage{subcaption}

\newcommand{\circledOne}{\text{\ding{172}}}
\newcommand{\circledTwo}{\text{\ding{173}}}
\newcommand{\circledThree}{\text{\ding{174}}}

\usepackage[breaklinks=true,colorlinks=true,linkcolor=red, citecolor=blue, urlcolor=blue]{hyperref}%

\usepackage{xcolor}
\newcommand{\al}[1]{{\color{blue}#1}} 

\newcommand{\dotprod}[2]{\langle #1,#2 \rangle} 

\begin{document}
\title{Stochastic Adversarial Noise in the ``Black Box'' Optimization Problem\thanks{The work was supported by the Ministry of Science and Higher Education of the Russian Federation (Goszadaniye) 075-00337-20-03, project No. 0714-2020-0005.}}
\titlerunning{Stochastic Adversarial Noise in the "Black Box" Optimization Problem}
%
\author{Aleksandr Lobanov\inst{1,2,3}\orcidID{0000-0003-1620-9581}}
%
\authorrunning{A. Lobanov}
%
\institute{Moscow Institute of Physics and Technology, Dolgoprudny, Russia
\\ \and 
 ISP RAS Research Center for Trusted Artificial Intelligence, Moscow, Russia \\ \and 
 Moscow Aviation Institute, Moscow, Russia\\
   \email{lobanov.av@mipt.ru}}

%
\maketitle              
\begin{abstract}
This paper is devoted to the study of the solution of a stochastic convex black box optimization problem. Where the black box problem means that the gradient-free oracle only returns the value of objective function, not its gradient. We consider non-smooth and smooth setting of the solution to the black box problem under adversarial stochastic noise. For two techniques creating gradient-free methods: smoothing schemes via $L_1$ and $L_2$ randomizations, we find the maximum allowable level of adversarial stochastic noise that guarantees convergence. Finally, we analyze the convergence behavior of the algorithms under the condition of a large value of noise level.

\end{abstract}
\section{Intoduction}
    
    The study of optimization problems, in which only a limited part of information is available, namely, the value of the objective function, in recent years is relevant and in demand. Such problems are usually classified as zero-order optimization problems \cite{Conn_2009} or black-box problems \cite{Audet_2017}. Where the latter intuitively understands that the black box is some process that has only two features (input, output), and about which nothing is known. The essential difference of this class of optimization problems is that the oracle returns only the value of the objective function and not its gradient or higher order derivatives, which are popular in numerical methods \cite{Polyak_1987}. This kind of oracle is commonly referred to as a zero-order/gradient-free oracle \cite{Rosenbrock_1960}, which acts as a black box. One of several leading directions for solving black-box optimization problems are zero-order numerical methods, which are based on first-order methods, approximating the gradient via finite-difference models. The first such method that gave rise to the field of research is the Kiefer-Wolfowitz method \cite{Kiefer_1952}, which was proposed in 1952. Since then, many gradient-free algorithms have been developed for applied problems in machine learning \cite{Papernot_2016,Papernot_2017}, distributed learning \cite{Scaman_2019,Yu_2021,Akhavan_2021}, federated learning \cite{Lobanov_2022,Patel_2022},  black-box attack to deep neural networks \cite{Chen_2017}, online \cite{Bach_2016,Shamir_2017,Gasnikov_2017} and standard \cite{Duchi_2015,Shibaev_2022,Gorbunov_2022} optimization. In particular, gradient-free methods are actively used  in the hyperparameters tuning of the deep learning model \cite{Li_2017,Hazan_2017,Elsken_2019}, as well as to solve the classical problem of adversarial multi-armed bandit \cite{Flaxman_2004,Bartlett_2008,Bubeck_2012}.
    
    In modern works \cite{Gasnikov_2022} authors try to develop optimal gradient-free algorithms according to three criteria at once: iteration complexity (the number of iterations performed successively guaranteeing convergence), oracle complexity (the total number of gradient-free oracle calls guaranteeing convergence), and the maximum admissible level of adversarial noise guaranteeing convergence. While the first two criteria seem understandable, the maximum level of adversarial noise can raise questions. However, in practice it is very common that the gradient-free oracle returns an inaccurate value of the objective function at the requested point. In other words, the gradient-free oracle outputs the value of the objective function with some adversarial noise. A simple example is rounding error (machine accuracy). There are also problems in which the computational complexity clearly depends on the level of adversarial noise, i.e., the greater the level of adversarial noise, the better the algorithm works in terms of computational complexity. This is why it is necessary to consider this criterion when creating a gradient-free optimization algorithm.
    
    In this paper we focus on solving a stochastic convex black-box optimization problem. We consider two black-box problem settings: smooth and non-smooth settings. We assume that the gradient-free oracle corrupted by an adversarial stochastic noise. Using smoothing schemes via $L_1$ and $L_2$ randomization, we derive the maximum permissible level of adversarial stochastic noise for each setting. And we also provide convergence results for the three-criteria optimal algorithm under adversarial stochastic noise, and provide discussions of how the algorithm would converge if the noise exceeded the optimal estimate.
    
    \subsection{Related Works}
    
     \textbf{Gradient approximation.}\;\;\; In many works \cite{Berahas_2022,Vaswani_2019,Hanzely_2020,Ermoliev_1976,Dvurechensky_2021,Polyak_1990,Akhavan_2020,Nemirovskij_1983,Gasnikov_ICML,Dvinskikh_2022,Akhavan_2022,Lobanov_2022} have developed methods, using various techniques to create algorithms via gradient approximation. For example, in \cite{Berahas_2022} a full gradient approximation instead of an exact gradient was used. Also in the smooth case, instead of an exact gradient, the some works use coordinate-wise randomization \cite{Vaswani_2019,Hanzely_2020} and random search randomization \cite{Ermoliev_1976,Dvurechensky_2021}. It is worth noting that these approximations assume that the gradient of the function is available. In \cite{Polyak_1990,Akhavan_2020}, the authors developed a gradient-free algorithm using a kernel approximation, which takes into account the increased smoothness of the objective function. For the non-smooth case \cite{Gasnikov_ICML,Dvinskikh_2022} describes a smoothing technique for creating gradient-free algorithms using $L_2$ randomization. And the paper \cite{Akhavan_2022} derived a better theoretical estimate of the variance of the gradient approximation, using $L_1$ randomization instead of $L_2$ randomization. In turn, the work \cite{Lobanov_2022} generalized the results of $L_1$ randomization to the non-smooth case and compared two smoothing schemes: $L_1$ and $L_2$ randomization, showing that in practice the clear advantage of $L_1$ randomization is not observed. In this paper, using smoothing schemes via $L_1$~and~$L_2$~randomizations, we derive optimal estimates of the level~of~noise, in which the optimality of oracle and iterative complexity~are~not~degraded.
    \newline
    \textbf{Adversarial noise.}\;\;\; There are many works \cite{Bayandina_2018,Beznosikov_2020,Risteski_2016,Vasin_2021,Dvinskikh_2022,Lobanov_2022,Akhavan_2021} that study optimization problems under adversarial noise. For example, the works \cite{Bayandina_2018,Beznosikov_2020} provided an optimal algorithm in terms of oracle complexity, but not optimal in terms of maximum allowable adversarial noise. Other works \cite{Risteski_2016,Vasin_2021} have proposed algorithms that are optimal in terms of the maximum permissible level of adversarial noise, but are not optimal according to the criterion of oracle complexity. A gradient-free algorithm, which is optimal according to two criteria: oracle complexity and the maximum permissible level of adversarial noise, is proposed in \cite{Dvinskikh_2022}. Whereas work \cite{Lobanov_2022} provided an optimal gradient-free algorithm for all three criteria: iteration complexity, oracle complexity, and the maximum allowable level of adversarial noise. However, these works considered the concept of adversarial deterministic noise, whereas for adversarial stochastic noise (see., e.g. \cite{Akhavan_2021}) the optimal bound for solving the black box optimization problem in smooth and non-smooth settings has not been obtained. In this paper, we will solve the black box problem in two settings with adversarial stochastic noise.
    
    \subsection{Paper Organization}
    The structure of this paper is as follows. In Section \ref{Section:Setting_and_Assumptions}, we describe the problem statement, as well as the basic assumptions and notations. We present the main result in Section \ref{Section:Main_Result}. In Section \ref{Section:Discussion}, we discuss the theoretical results obtained. While Section \ref{Section:Conclusion} concludes this paper.

\section{Setting and Assumptions}\label{Section:Setting_and_Assumptions}
    We study a standard stochastic convex black-box optimization problem:
    \begin{equation}\label{Init_problem}
        f^* = \min_{x \in Q} \left\{ f(x) := \mathbb{E}_\xi \left[ f(x,\xi) \right] \right\},
    \end{equation}
    where $Q \subseteq \mathbb{R}^d$ is a convex and compact set, $f: Q \rightarrow \mathbb{R}$ is  convex function. Since problem \eqref{Init_problem} is a general problem formulation, we introduce standard assumptions and definitions that will narrow down the class of problems~under consideration.
    
    \begin{definition}[Gradient-free oracle]\label{def:GFOracle}
        Gradient-free oracle returns a function value $f(x,\xi)$ at the requested point $x$ with some adversarial stochastic noise, i.e. for all $x \in Q$
        \begin{equation*}
            f_\delta(x, \xi) := f(x,\xi) + \delta.
        \end{equation*}
    \end{definition}
    \begin{assumption}[Lipschitz continuity of objective function]\label{ass:Lipschitz_continuity}
        The function $f(x, \xi)$ is an $M$-Lipschitz continuous function in the $l_p$-norm, i.e for all $x, y \in Q$ we have
        \begin{equation*}
            |f(y,\xi) - f(x, \xi)| \leq M(\xi) \| y-x \|_p.
        \end{equation*}
        Moreover, there is a positive constant $M$, which is defined in the following way: $\mathbb{E} \left[ M^2(\xi) \right]\leq M^2$. In particular, for $p = 2$ we use the notation $M_2$ for the Lipschitz constant.
    \end{assumption}
    \begin{assumption}[Convexity on the set $Q_\gamma$]\label{ass:convexity_Q_gamma} 
    Let $\gamma > 0$ a small number to be defined later and $Q_\gamma := Q + B_p^d(\gamma)$, then the function $f$ is convex on the set $Q_\gamma$.
    \end{assumption}
    The following assumption we need to solve problem \eqref{Init_problem} in a smooth setting.
    \begin{assumption}[Smoothness of function]\label{ass:Smoothness_of_function}
        The function $f$ is smooth, that is, differentiable on $Q$ and such that for all $x,y \in Q$ with $L > 0$ we have
        \begin{equation*}
            \| \nabla f(y) - \nabla f(x) \|_q \leq L \| y - x \|_p.
        \end{equation*}
    \end{assumption}
    Next, we introduce the assumption about adversarial noise.
    \begin{assumption}[Adversarial noise]\label{ass:Adversarial_noise}
        It holds, that the random variables $\delta_1$ and $\delta_2$ are independent from $e \in S_p^d(1)$ as well as $\mathbb{E}\left[ \delta_1^2 \right] \leq \Delta^2$ and $\mathbb{E}\left[ \delta_2^2 \right]~\leq~\Delta^2$.
    \end{assumption}
    
    Our Assumption \ref{ass:Lipschitz_continuity} is necessary for theoretical proofs in each setting: smooth and non-smooth. This Assumption \ref{ass:Lipschitz_continuity} is common in literature (see e.g. \cite{Gasnikov_ICML,Lobanov_2023}). Assumption \ref{ass:convexity_Q_gamma} is standard for works using smoothing technique (see e.g.~\cite{Risteski_2016,Lobanov_2022}). Assumption \ref{ass:Smoothness_of_function} was introduced only for smooth tuning, and is also often found in the literature (e.g., in previous works such as \cite{Bach_2016,Akhavan_2020}). Whereas Definition \ref{def:GFOracle} is similar to \cite{Gasnikov_2022}, only using adversarial stochastic noise instead of adversarial deterministic noise. Finally, Assumption \ref{ass:Adversarial_noise} is the same as in the previous work~\cite{Akhavan_2022}.
    
    \subsection*{Notation}
    We use $\dotprod{x}{y}:= \sum_{i=1}^{d} x_i y_i$ to denote standard inner product of $x,y \in \mathbb{R}^d$, where $x_i$ and $y_i$ are the $i$-th component of $x$ and $y$ respectively. We denote $l_p$-norms (for~$p \geq 1$) in $\mathbb{R}^d$ as $\| x\|_p := \left( \sum_{i=1}^d |x_i|^p \right)^{1/p}$. Particularly for $l_2$-norm in $\mathbb{R}^d$ it follows $\| x \|_2 := \sqrt{\dotprod{x}{x}}$. We denote $l_p$-ball as $B_p^d(r):=\left\{ x \in \mathbb{R}^d : \| x \|_p \leq r \right\}$ and $l_p$-sphere as $S_p^d(r):=\left\{ x \in \mathbb{R}^d : \| x \|_p = r \right\}$. Operator $\mathbb{E}[\cdot]$ denotes full mathematical expectation. To denote the distance between the initial point $x^0$ and the solution of the initial problem $x_*$ we introduce $R := \mathcal{\tilde{O}}\left( \| x^0 - x_*\|_p\right)$, where we notation $\mathcal{\tilde{O}} (\cdot)$ to hide logarithmic factors.

\section{Main Result}\label{Section:Main_Result}
    In this section, we build our narrative on solving the black-box optimization problem \eqref{Init_problem} in a non-smooth setting. We will discuss the smooth setting as a special case of the non-smooth setting (presence of Assumption \ref{ass:Smoothness_of_function}) at the end of the section in Remark \ref{Remark}. This section is organized as follows: in Subsection~\ref{subsection:L1_random} we introduce the smooth approximation of a non-smooth function and its properties, the gradient approximation by $L_1$ randomization and its properties, i.e. we describe the smoothing scheme via $L_1$ randomization. In Subsection \ref{subsection:L2_random}, we do the same and describe the smoothing scheme via $L_2$ randomization. And in Subsection~\ref{subsection:Maximum_level_noise} we present maximum allowed level of adversarial stochastic~noise. So we begin by describing the smoothing technique via $L_1$ randomization.
    
    \subsection{Smoothing scheme via $L_1$ randomization}\label{subsection:L1_random}
    Since problem \eqref{Init_problem} is non-smooth, we introduce the following smooth approximation of the non-smooth function:
    \begin{equation}
        \label{f_gamma}
        f_\gamma(x) := \mathbb{E}_{\Tilde{e}} \left[ f(x + \gamma \Tilde{e})\right],
    \end{equation}
    where $\gamma>0$ is a smoothing parameter, $\Tilde{e}$ is a random vector uniformly distributed on $B_{1}^d(1)$. Here $f(x):= \mathbb{E}\left[ f(x,\xi) \right]$. The following lemma provides the connection between the smoothed and the original function.
    \begin{lemma}\label{Lemma:L1_connect_f_with_f_gamma}
        Let Assumptions \ref{ass:Lipschitz_continuity}, \ref{ass:convexity_Q_gamma} it holds, then for all $x \in Q$ we have
        \begin{equation*}
            f(x)\leq f_\gamma(x) \leq f(x) + \frac{2}{\sqrt{d}}\gamma M_2.
        \end{equation*}
    \end{lemma}
    \begin{proof}
    For the first inequality we use the convexity of the function $f(x)$
    \begin{equation*}
        f_\gamma(x) = \mathbb{E}_{\Tilde{e}} \left[ f(x + \gamma \Tilde{e}) \right] \geq \mathbb{E}_{\Tilde{e}} \left[ f(x) + \dotprod{\nabla f(x)}{\gamma \Tilde{e}}) \right] = \mathbb{E}_{\Tilde{e}} \left[ f(x) \right] = f(x).
    \end{equation*}
    For the second inequality, applying Lemma 1 of \cite{Akhavan_2022} = \circledOne, we have
    \begin{eqnarray*}
        | f_\gamma (x)- f(x) | = | \mathbb{E}_{\Tilde{e}} \left[ f(x + \gamma \Tilde{e}) \right] - f(x) | &\leq& \mathbb{E}_{\Tilde{e}} \left[ | f(x + \gamma \Tilde{e}) - f(x) | \right]\\
        &\leq& \gamma M_2 \mathbb{E}_{\Tilde{e}} \left[ \| \Tilde{e} \|_2 \right] \overset{\circledOne}{\leq} \frac{2}{\sqrt{d}} \gamma M_2,
    \end{eqnarray*}

    using the fact that $f$ is $M_2$-Lipschitz function.
    \end{proof}\qed    
    
    The following lemmas confirm that the Lipschitz continuity property holds and provide the Lipschitz constant of gradient for the smoothed function.
    
    \begin{lemma}\label{Lemma:L1_M_lipschitz_continuity}
        Let Assumptions \ref{ass:Lipschitz_continuity}, \ref{ass:convexity_Q_gamma} it holds, then for $f_\gamma(x)$ from \eqref{f_gamma} we have
        \begin{equation*}
            |f_\gamma(y) - f_\gamma(x) | \leq M \| y - x \|_p, \;\;\; \forall x,y \in Q.
        \end{equation*}
    \end{lemma}
    \begin{proof} Using $M$-Lipschitz continuity of function $f$ we obtain
        \begin{eqnarray*}
            | f_\gamma (y)- f_\gamma(x) | \leq \mathbb{E}_{\Tilde{e}} \left[ | f(y + \gamma \Tilde{e})  - f(x + \gamma \Tilde{e}) | \right] \leq M \| y - x \|_p.
        \end{eqnarray*}
    \end{proof}\qed
    
    \begin{lemma}[Lemma 1, \cite{Lobanov_2022}]\label{Lemma:Lipschitz_gradient}
        Let Assumptions \ref{ass:Lipschitz_continuity}, \ref{ass:convexity_Q_gamma} it holds, then $f_\gamma(x)$ has $L_{f_\gamma} = \frac{dM}{\gamma}$-Lipschitz gradient
        \begin{equation*}
            \| \nabla f_\gamma(y) - \nabla f_\gamma(x) \|_q \leq L_{f_{\gamma}} \| y - x \|_p, \;\;\; \forall x,y \in Q.
        \end{equation*}
    \end{lemma}
    
    The gradient of $f_\gamma(x,\xi)$ can be estimated by the following approximation:
    \begin{equation}\label{approximation_gradient_L_1}
        \nabla f_\gamma(x, \xi, e) = \frac{d}{2 \gamma} \left( f_{\delta_1}(x+ \gamma e, \xi) - f_{\delta_2}(x - \gamma e, \xi) \right) \text{sign}(e),
    \end{equation}
    where $f_\delta(x,\xi)$ is gradient-free oracle from Definition \ref{def:GFOracle}, $e$ is a random vector uniformly distributed on $S_1^d(\gamma)$. The following lemma provides properties of the gradient $\nabla f_\gamma(x,\xi,e)$.
    
    \begin{lemma}[Lemma 4, \cite{Akhavan_2022}]\label{Lemma:L1_sigma}
         Gradient $\nabla f_\gamma(x,\xi,e)$ has bounded variance (second moment) for all $x \in Q$
        \begin{equation*}
           \mathbb{E}_{\xi, e} \left[ \| \nabla f_\gamma (x, \xi, e) \|^2_q \right] \leq \kappa(p,d) \left( M_2^2 + \frac{d^2 \Delta^2}{12(1+\sqrt{2})^2 \gamma^2}  \right), 
        \end{equation*}
        where $1/p + 1/q = 1$ and 
        \begin{equation*}
            \kappa(p,d) = \kappa(p,d)  = 48(1+\sqrt{2})^2 d^{2 - \frac{2}{p}}.
        \end{equation*}
    \end{lemma}
    
    \subsection{Smoothing scheme via $L_2$ randomization}\label{subsection:L2_random}
    Since problem \eqref{Init_problem} is non-smooth, we introduce the following smooth approximation of the non-smooth function:
    \begin{equation}
        \label{L2_f_gamma}
        \tilde{f}_\gamma(x) := \mathbb{E}_{\Tilde{e}} \left[ f(x + \gamma \Tilde{e})\right],
    \end{equation}
    where $\gamma>0$ is a smoothing parameter, $\Tilde{e}$ is a random vector uniformly distributed on $B_{2}^d(1)$. Here $f(x):= \mathbb{E}\left[ f(x,\xi) \right]$. The following lemma provides the connection between the smoothed and the original function.
    \begin{lemma}\label{Lemma:L2_connect_f_with_f_gamma}
        Let Assumptions \ref{ass:Lipschitz_continuity}, \ref{ass:convexity_Q_gamma} it holds, then for all $x \in Q$ we have
        \begin{equation*}
            f(x)\leq \tilde{f}_\gamma(x) \leq f(x) + \gamma M_2.
        \end{equation*}
    \end{lemma}
    \begin{proof}
    For the first inequality we use the convexity of the function $f(x)$
    \begin{equation*}
        \tilde{f}_\gamma(x) = \mathbb{E}_{\Tilde{e}} \left[ f(x + \gamma \Tilde{e}) \right] \geq \mathbb{E}_{\Tilde{e}} \left[ f(x) + \dotprod{\nabla f(x)}{\gamma \Tilde{e}}) \right] = \mathbb{E}_{\Tilde{e}} \left[ f(x) \right] = f(x).
    \end{equation*}
    For the second inequality we have
    \begin{eqnarray*}
        | \tilde{f}_\gamma (x)- f(x) | = | \mathbb{E}_{\Tilde{e}} \left[ f(x + \gamma \Tilde{e}) \right] - f(x) | &\leq& \mathbb{E}_{\Tilde{e}} \left[ | f(x + \gamma \Tilde{e}) - f(x) | \right]\\
        &\leq& \gamma M_2 \mathbb{E}_{\Tilde{e}} \left[ \| \Tilde{e} \|_2 \right] \leq \gamma M_2,
    \end{eqnarray*}

    using the fact that $f$ is $M_2$-Lipschitz function.
    \end{proof}\qed    
    
    The following lemmas confirm that the Lipschitz continuity property holds and provide the Lipschitz constant of gradient for the smoothed function.
    
    \begin{lemma}\label{Lemma:L1_M_lipschitz_continuity}
        Let Assumptions \ref{ass:Lipschitz_continuity}, \ref{ass:convexity_Q_gamma} it holds, then for $\tilde{f}_\gamma(x)$ from \eqref{f_gamma} we have
        \begin{equation*}
            |\tilde{f}_\gamma(y) - \tilde{f}_\gamma(x) | \leq M \| y - x \|_p, \;\;\; \forall x,y \in Q.
        \end{equation*}
    \end{lemma}
    \begin{proof} Using $M$-Lipschitz continuity of function $f$ we obtain
        \begin{eqnarray*}
            | \tilde{f}_\gamma (y)- \tilde{f}_\gamma(x) | \leq \mathbb{E}_{\Tilde{e}} \left[ | f(y + \gamma \Tilde{e})  - f(x + \gamma \Tilde{e}) | \right] \leq M \| y - x \|_p.
        \end{eqnarray*}
    \end{proof}\qed
    
    \begin{lemma}[Theorem 1, \cite{Gasnikov_2022}]\label{Lemma:L2_Lipschitz_gradient}
        Let Assumptions \ref{ass:Lipschitz_continuity}, \ref{ass:convexity_Q_gamma} it holds, then $\tilde{f}_\gamma(x)$ has $L_{\tilde{f}_\gamma} = \frac{\sqrt{d}M}{\gamma}$-Lipschitz gradient
        \begin{equation*}
            \| \nabla \tilde{f}_\gamma(y) - \nabla \tilde{f}_\gamma(x) \|_q \leq L_{\tilde{f}_{\gamma}} \| y - x \|_p, \;\;\; \forall x,y \in Q.
        \end{equation*}
    \end{lemma}
    
    The gradient of $\tilde{f}_\gamma(x,\xi)$ can be estimated by the following approximation:
    \begin{equation}\label{approximation_gradient_L_2}
        \nabla \tilde{f}_\gamma(x, \xi, e) = \frac{d}{2 \gamma} \left( f_{\delta_1}(x+ \gamma e, \xi) - f_{\delta_2}(x - \gamma e, \xi) \right) e,
    \end{equation}
    where $f_\delta(x,\xi)$ is gradient-free oracle from Definition \ref{def:GFOracle}, $e$ is a random vector uniformly distributed on $S_2^d(\gamma)$. The following lemma provides properties of the gradient $\nabla \tilde{f}_\gamma(x,\xi,e)$.
    
    \begin{lemma}[\cite{Shamir_2017,Lobanov_2022}]\label{Lemma:L2_sigma}
         Gradient $\nabla \tilde{f}_\gamma(x,\xi,e)$ has bounded variance (second moment) for all $x \in Q$
        \begin{equation*}
           \mathbb{E}_{\xi, e} \left[ \| \nabla \tilde{f}_\gamma (x, \xi, e) \|^2_q \right] \leq \kappa(p,d) \left( d M_2^2 + \frac{d^2 \Delta^2}{\sqrt{2} \gamma^2}  \right), 
        \end{equation*}
        where $1/p + 1/q = 1$ and 
        \begin{equation*}
            \kappa(p,d) = \sqrt{2} \min \left\{ q, \ln d \right\} d^{1-\frac{2}{p}}.
        \end{equation*}
    \end{lemma}

    \subsection{Maximum level of adversarial stochastic noise}\label{subsection:Maximum_level_noise}
    In this subsection, we present our main result, namely the optimal bounds in terms of the maximum allowable level of the adversarial stochastic noise for smoothing techniques discussed in Subsections \ref{subsection:L1_random} and \ref{subsection:L2_random}. Next, we will consider case when $\Delta > 0$, i.e., there is adversarial noise. Then, before writing down main theorem, let us show that the gradient approximations \eqref{approximation_gradient_L_1} and \eqref{approximation_gradient_L_2} are unbiased.
    
    \begin{itemize}
        \item[$\bullet$] Gradient approximation \eqref{approximation_gradient_L_1} is unbiased:
        \begin{eqnarray*}
            \mathbb{E}_{e,\xi}\left[ \nabla f_\gamma(x, \xi, e)\right] &=& \mathbb{E}_{e,\xi}\left[\frac{d}{2 \gamma} \left( f_{\delta_1}(x+ \gamma e, \xi) - f_{\delta_2}(x - \gamma e, \xi) \right) \text{sign}(e)\right]\\
            &=& \mathbb{E}_{e,\xi}\left[\frac{d}{2 \gamma} \left( f(x+ \gamma e, \xi) + \delta_1 - f(x - \gamma e, \xi) - \delta_2 \right) \text{sign}(e)\right]\\
            &\overset{\circledTwo}{=}& \mathbb{E}_{e,\xi}\left[\frac{d}{2 \gamma} \left( f(x+ \gamma e, \xi) - f(x - \gamma e, \xi) \right) \text{sign}(e)\right]\\
            &=& \mathbb{E}_{e}\left[\frac{d}{2 \gamma} \left( f(x+ \gamma e) - f(x - \gamma e) \right) \text{sign}(e)\right]\\
            &\overset{\circledThree}{=}& \nabla f_\gamma(x),
        \end{eqnarray*}
        where $\circledTwo$ = we assume that Assumption \ref{ass:Adversarial_noise} is satisfied, $\circledThree$ = Lemma 1 \cite{Akhavan_2022}.
        
        \item[$\bullet$] Gradient approximation \eqref{approximation_gradient_L_2} is unbiased:
        \begin{eqnarray*}
            \mathbb{E}_{e,\xi}\left[ \nabla \tilde{f}_\gamma(x, \xi, e)\right] &=& \mathbb{E}_{e,\xi}\left[\frac{d}{2 \gamma} \left( f_{\delta_1}(x+ \gamma e, \xi) - f_{\delta_2}(x - \gamma e, \xi) \right) e\right]\\
            &=& \mathbb{E}_{e,\xi}\left[\frac{d}{2 \gamma} \left( f(x+ \gamma e, \xi) + \delta_1 - f(x - \gamma e, \xi) - \delta_2 \right) e\right]\\
            &\overset{\circledTwo}{=}& \mathbb{E}_{e,\xi}\left[\frac{d}{2 \gamma} \left( f(x+ \gamma e, \xi) - f(x - \gamma e, \xi) \right) e\right]\\
            &=& \mathbb{E}_{e}\left[\frac{d}{2 \gamma} \left( f(x+ \gamma e) - f(x - \gamma e) \right) e\right]\\
            &\overset{\circledThree}{=}& \nabla \tilde{f}_\gamma(x),
        \end{eqnarray*}
        where $\circledTwo$ = we assume that Assumption \ref{ass:Adversarial_noise} is satisfied, $\circledThree$ = Theorem 2.2 \cite{Gasnikov_ICML}.
    \end{itemize}
    
    Since the adversarial noise does not accumulate in the bias (since at $\Delta>0$ the gradient approximation is unbiased), the maximum allowable level of adversarial stochastic noise will only accumulate in the variance. Then the following Theorem \ref{Theorem} presents the optimal estimates for adversarial noise.
    
    \begin{theorem}\label{Theorem}
        Let Assumptions \ref{ass:Lipschitz_continuity},\ref{ass:convexity_Q_gamma},\ref{ass:Adversarial_noise} be satisfied, then the algorithm \textbf{A}$(L, \sigma^2)$ obtained by applying smoothing schemes (see Subsections \ref{subsection:L1_random} and \ref{subsection:L2_random}) based on the first order method 
        \begin{enumerate}
            \item for Smoothing scheme via $L_1$ randomization has level of adversarial noise
            \begin{equation*}
                \Delta \lesssim \frac{\varepsilon}{\sqrt{d}};
            \end{equation*}
            
            \item for Smoothing scheme via $L_2$ randomization has level of adversarial noise
            
            \begin{equation*}
                \Delta \lesssim \frac{\varepsilon}{\sqrt{d}}, 
            \end{equation*}
        \end{enumerate}
        where $\varepsilon$ is accuracy solution to problem \eqref{Init_problem}, $\mathbb{E}[f(x_N)] - f^* \leq \varepsilon$.
    \end{theorem}
    \begin{proof}
        Since the adversarial noise accumulates only in the variance, in order to guarantee convergence (without losing in the oracle complexity estimates) it is necessary to guarantee that the following inequality is satisfied:
        \begin{itemize}
            \item[$\bullet$] for Smoothing scheme via $L_1$ randomization from Lemma \eqref{Lemma:L1_sigma} 
            \begin{equation*}
                M_2^2 \geq \frac{d^2 \Delta^2} {12 (1+\sqrt{2})^2 \gamma^{2} }
            \end{equation*}
            Then we have, using the fact that $\gamma = \frac{\sqrt{d} \varepsilon}{2 M_2}$
            \begin{equation*}
                \Delta \leq \sqrt{\frac{12 (1+\sqrt{2})^2 M_2^2 \gamma^2}{d^2}} = \frac{2 \sqrt{3} (1+\sqrt{2}) M_2 \sqrt{d} \varepsilon}{d M_2} \simeq \frac{\varepsilon}{\sqrt{d}}.
            \end{equation*}
            
            \item[$\bullet$] for Smoothing scheme via $L_2$ randomization from Lemma \eqref{Lemma:L2_sigma} 
            \begin{equation*}
                d M_2^2 \geq \frac{d^2 \Delta^2} {\sqrt{2} \gamma^{2} }
            \end{equation*}
            Then we have, using the fact that $\gamma = \frac{\varepsilon}{2 M_2}$
            \begin{equation*}
                \Delta \leq \sqrt{\frac{\sqrt{2}d M_2^2 \gamma^2}{d^2}} = \frac{2^{1/4} M_2 \varepsilon}{\sqrt{d} M_2} \simeq \frac{\varepsilon}{\sqrt{d}}.
            \end{equation*}
        \end{itemize}
    \end{proof}
    The results of Theorem \ref{Theorem} show that the maximum allowable level of adversarial stochastic noise is the same for the two smoothing schemes. Moreover, this estimation guarantees convergence without losing in other criteria, i.e. if we take as Algorithm \textbf{A}$(L, \sigma^2)$ the accelerated batched first-order method and apply one of the two smoothing techniques, we will create an optimal algorithm for three criteria at once: oracle complexity, number of successive iterations, and the maximum permissible level of adversarial stochastic noise. An explanation of the choice of the smoothing parameter value can be found in Corollary 1 \cite{Lobanov_2022}.
    
    \begin{remark}[Smoothing setting]\label{Remark}
        Since Assumption \ref{ass:Smoothness_of_function} is satisfied in the smooth statement of problem \eqref{Init_problem}, all statements above hold except Lemmas \ref{Lemma:L1_connect_f_with_f_gamma} and \ref{Lemma:L2_connect_f_with_f_gamma}: $f(x)\leq \tilde{f}_\gamma(x) \leq f(x) + \al{\frac{2}{d}\gamma^2 L^2}$ (in Lemma \ref{Lemma:L1_connect_f_with_f_gamma}) and $f(x)\leq \tilde{f}_\gamma(x) \leq f(x) + \al{\gamma^2 L^2}$ (in~Lemma \ref{Lemma:L2_connect_f_with_f_gamma}). Thus, we can conclude that if the Assumptions \ref{ass:Lipschitz_continuity}-\ref{ass:Adversarial_noise} are satisfied, then based on the accelerated first-order batched algorithm \textbf{A}$(\sigma^2)$, and using any gradient approximation ($L_1$~or~$L_2$ randomization) it will allow to create an optimal algorithm according to the three criteria, where the maximum allowed level of adversarial stochastic noise~is~$\Delta~\lesssim~\sqrt{\frac{\varepsilon}{d}}$.
        
    \end{remark}

\section{Discussion}\label{Section:Discussion}
    The essential difference between the stochastic adversarial noise considered in this paper and the deterministic adversarial noise is that in our case the adversarial noise does not accumulate into a bias, while in the deterministic noise concept the opposite is true. Precisely because this concept behaves less adversely, it is possible to solve the problem with a large value of adversarial noise without losing convergence, unlike deterministic adversarial noise, which has a maximum allowable noise level equal to $\mathcal{O} \left( \varepsilon^2 d^{-1/2} \right)$. In order to achieve optimality on the three criteria, it is necessary to rely on an accelerated (for optimal estimation of oracle complexity $\sim \mathcal{O} \left( d \varepsilon^{-2} \right)$ for case $p=2$) batched (for optimal estimation of iterative complexity $\sim \mathcal{O} \left( d^{1/4} \varepsilon^{-2} \right)$) first-order method. And also using the concept of adversarial stochastic noise, by guaranteeing the prevalence of variance (e.g. from \eqref{Lemma:L2_sigma}:  $d M_2^2 \geq \frac{d^2 \Delta^2} {\sqrt{2} \gamma^{2} }$ ), the Theorem \ref{Theorem} guarantees optimal convergence in three criteria. However, if the evaluation of the second moment is dominated by stochastic adversarial noise (i.e. $d M_2^2 < \frac{d^2 \Delta^2} {\sqrt{2} \gamma^{2} }$), then the convergence of the algorithm will worsen (estimation of oracle complexity will be $\sim \varepsilon^{-4}$), since the second moment will be $\kappa(p,d) d^2 \Delta^2 \varepsilon^{-2}$, considering that $\gamma \sim \varepsilon$. Such a low estimate of oracle complexity corresponds to oracle complexity when a gradient approximation look like $\frac{d}{\gamma}f(x + \gamma e) e$ (e.g. for $L_2$ randomization), i.e. when gradient approximation requires only one gradient-free oracle calls.

\section{Conclusion}\label{Section:Conclusion}
In this paper, we studied stochastic convex black box optimization problems in two settings: non-smooth and smooth settings. For two smoothing techniques (with $L_1$ and $L_2$ randomization), we obtained the maximum allowable levels of adversarial stochastic noise. We also showed in this paper that for using any smoothing scheme via $L_1$ or $L_2$ randomization, as well as solving the black box problem in a non-smooth or smooth setting, the maximum value of adversarial noise is the same. Finally, we analyzed the convergence rate of the algorithm, provided that the noise level is large.

%
%
%
%

\end{document}